\documentclass{article}
\usepackage{graphicx} 
\usepackage{amsthm, amsmath, amssymb, tikz, bm}
\usepackage{mathtools}
\usepackage{xifthen}
\usepackage{comment}
\usepackage[margin=4.5cm]{geometry}
\usepackage[shortlabels]{enumitem}
\usepackage{todonotes}
\usepackage[colorlinks=true,
linkcolor=blue,citecolor=blue,
urlcolor=blue]{hyperref}

\usetikzlibrary{calc,shapes, backgrounds}
\usetikzlibrary[patterns]

\newtheorem{theorem}{Theorem}[section]
\newtheorem{lemma}[theorem]{Lemma}
\newtheorem{sublemma}{}[theorem]
\newtheorem{corollary}[theorem]{Corollary}

\newtheorem{proposition}[theorem]{Proposition}

\newcommand{\romannum}[1]{\romannumeral#1\relax}

\newcommand{\co}{\operatorname{co}}
\newcommand{\si}{\operatorname{si}}
\newcommand{\cl}{\operatorname{cl}}
\newcommand{\W}{\mathcal{W}}
\makeatletter
\newcommand*{\rom}[1]{\expandafter\@slowromancap\romannumeral #1@}
\makeatother
\date{}
\title{Super-minimally $3$-connected matroids}
\author{
Wayne Ge \thanks{Louisiana State University, Baton Rouge, LA, 70803
({\tt yge4@lsu.edu}).}
\and
James Oxley \thanks{Louisiana State University, Baton Rouge, LA, 70803
({\tt oxley@math.lsu.edu}).}
}

\begin{document}

\maketitle
\begin{abstract}
A super-minimally $k$-connected matroid is a $k$-connected matroid having no proper $k$-connected restriction of size at least $2k-2$. This extends the corresponding concept for graphs. For $k=2$ and $k=3$, we determine the maximum size of a super-minimally $k$-connected rank-$r$ matroid and characterize, in each case, those matroids attaining the extremal bound. These results parallel Murty's results for minimally $2$-connected matroids and Oxley's results for minimally $3$-connected matroids. 
\end{abstract}

\section{Introduction}
The study of connectivity plays a central role in both graph theory and matroid theory, particularly in understanding the structure of objects that are minimally sufficient to maintain a given level of connectivity. Classical results of Dirac~\cite[(7)]{Dirac} and Halin~\cite[(7.6)]{Halin-density} established sharp bounds on the size of minimally $k$-connected graphs for $k\in\{2,3\}$, and these ideas were later extended to matroids by Murty~\cite[Theorems 3.2 and 3.4]{Murty} and Oxley~\cite[(4.7)]{Oxley_OMC}, respectively.

Recently, Ge~\cite{Ge_sm3c} introduced the class of {\it super-minimally $k$-connected graphs}, namely $k$-connected graphs that contain no proper $k$-connected subgraphs, and determined tight bounds on their density for $k\in\{2,3\}$. This raises the natural question as to whether analogous extremal phenomena occur for matroids under a similar strengthening of minimal connectivity.

Motivated by this question, we introduce the class of super-minimally $k$-connected matroids. Since a matroid with fewer than $2k-2$ elements cannot have a $(k-1)$-separation, we define a $k$-connected matroid $M$ to be {\it super-minimally $k$-connected} if it has no proper $k$-connected restriction with at least $2k-2$ elements. This notion captures matroids whose connectivity is highly fragile in the sense that deleting every subset of the ground set destroys $k$-connectivity. When $k=2$, we make the following elementary observation.

\begin{proposition}
    A matroid $M$ is super-minimally $2$-connected if and only if $M$ is isomorphic to $U_{1,1}$, or $U_{r,r+1}$ for some integer $r\geq 0$. In particular, $|E(M)|\leq r(M)+1$.
\end{proposition}
    
Our main result is the following analogue of Oxley's theorem~\cite[(4.7)]{Oxley_OMC} on the density of minimally $3$-connected matroids. It extends Ge's bound~\cite[Theorem~1.10]{Ge_sm3c} for super-minimally $3$-connected graphs to matroids.

\begin{theorem}\label{thm:density_sm3c_matroids}
Let $M$ be a super-minimally $3$-connected matroid with at least four elements. Then
\[
|E(M)| \le 2r(M).
\]
Moreover, equality holds if and only if $M \cong U_{2,4}$, or $M \cong M(\W_k)$ or $\W^k$ for some $k \ge 3$.
\end{theorem}

In the final section, we derive results on triads in super-minimally $3$-connected matroids from Lemos's results on minimally $3$-connected matroids~\cite{Lemos2004,Lemos2007}.

\section{Preliminaries}
\subsection{Some connectivity results}
We assume familiarity with the basic notions of matroid theory as found, for example, in~\cite{Oxl11}. In particular, Chapter 8 of~\cite{Oxl11} gives a comprehensive treatment of matroid connectivity. Let $M$ be a matroid. We use $\lambda_M$ to denote the connectivity function of $M$ and will often abbreviate it as $\lambda$. We omit the elementary proof of the following.

\begin{proposition}\label{prop:closure coclosure}
    Suppose $\{X,Y\}$ is a $k$-separation of a matroid $M$ and $e$ is an element of $M$ in $\cl(X)\cup\cl^*(X)$. Then $\{X\cup e,Y-e\}$ is a $k$-separation of $M$ if and only if $|Y-e|\geq k$.
\end{proposition}

The following lemma, due to Tutte~\cite[(7.2)]{Tutte_matroids}, has been used extensively in the study of $3$-connected matroids and is often referred to as Tutte’s Triangle Lemma (see, for example,~\cite[Lemma 8.7.7]{Oxl11}).

\begin{lemma}
    Let $M$ be a $3$-connected matroid having at least four elements and suppose that $\{e,f,g\}$ is a triangle of $M$ such that neither $M\setminus e$ nor $M\setminus f$ is $3$-connected. Then $M$ has a triad that contains $e$ and exactly one of $f$ and $g$.
\end{lemma}

The next lemma is commonly referred to as Bixby’s Lemma~\cite{Bixby} (see also~\cite[Lemma~8.7.3]{Oxl11}). Let $M$ be a matroid. We use $\si(M)$ to denote the simple matroid associated with $M$. The {\it cosimplification}, $\co(M)$, of $M$ is $(\si(M^*))^*$.

\begin{lemma}
    Let $e$ be an element of a $3$-connected matroid $M$. Then $\si(M/e)$ is $3$-connected or $\co(M\setminus e)$ is $3$-connected.
\end{lemma}

One of the main tools in the study of $3$-connected matroids is Tutte's Wheels-and-Whirls Theorem. When $M$ is a $3$-connected matroid, an element $e$ is {\it essential} if neither $M\backslash e$ nor $M/e$ is $3$-connected. The next result follows immediately from an extension of Tutte's Wheels-and-Whirls Theorem that is due to Oxley and Wu~\cite[Theorem 1.6]{Oxley_Wu} (see also~\cite[Lemma 8.8.6]{Oxl11}).

\begin{theorem}\label{thm:two non-essential}
    Suppose $M$ is a $3$-connected matroid with at least four elements that is not isomorphic to a whirl or to the cycle matroid of a wheel. Then $M$ has at least two nonessential elements.
\end{theorem}

The following lemmas will be useful in our arguments.

\begin{lemma}~\cite[(2.6)]{Oxley_OMC}\label{lemma:oxley contraction}
    Suppose that $x$ and $y$ are distinct elements of an $n$-connected matroid $M$ where $n\geq 2$ and $|E(M)|\geq 2(n-1)$. Assume that $M\setminus x/y$ is $n$-connected but that $M\setminus x$ is not $n$-connected. Then $M$ has a cocircuit of size $n$ containing $x$ and $y$.
\end{lemma}

\begin{lemma}~\cite[(5.2)]{Oxley_OMC}\label{lemma:small rank wheel and whirl}
    Let $M$ be a rank-$r$ minimally $3$-connected matroid with precisely $2r$ elements. If $3\leq r\leq 5$, then $M$ is isomorphic to either $M(\W_r)$ or $\W^r$.
\end{lemma}

\subsection{Small $3$-connected matroids}
Recall that a $3$-connected matroid $M$ is super-minimally $3$-connected if $M$ does not have a proper $3$-connected restriction of size at least four. Therefore, if the size of $M$ is at most four, then $M$ is super-minimally $3$-connected if and only if $M$ is $3$-connected. Table~\ref{table:all sm3c matroids} lists all $3$-connected matroids on at most four elements.

\begin{center}
\begin{table}[h!]
\centering
\begin{tabular}{|c|l|}
\hline
\textbf{Number $n$ of elements} & \textbf{$3$-connected $n$-element matroids} \\ \hline
1                               & $U_{0,1}$, $U_{1,1}$                        \\ \hline
2                               & $U_{1,2}$                                   \\ \hline
3                               & $U_{1,3}$, $U_{2,3}$                        \\ \hline
4                               & $U_{2,4}$                                   \\ \hline
\end{tabular}
\caption{All $3$-connected matroids on at most four elements.}
\label{table:all sm3c matroids}
\end{table}
\end{center}

Although every super-minimally $3$-connected graph is minimally $3$-connected~\cite[Lemma~1.1]{Ge_sm3c}, a super-minimally $3$-connected matroid is not necessarily minimally $3$-connected. For instance, the uniform matroid $U_{1,3}$ is super-minimally $3$-connected but not minimally $3$-connected, since each single-element deletion of it is isomorphic to $U_{1,2}$, which is $3$-connected. The next proposition shows that this counterintuitive phenomenon occurs only when $|E(M)|$ is small. We omit its elementary proof.

\begin{proposition}\label{prop:small smkc}
    Let $k\geq 2$ be an integer and let $M$ be a matroid.
    \begin{enumerate}
        \item[(\romannum{1})] If $|E(M)|\leq 2k-2$, then $M$ is super-minimally $k$-connected if and only if $M$ is $k$-connected.
        \item[(\romannum{2})] If $M$ is super-minimally $k$-connected and $|E(M)|>2k-2$, then $M$ is minimally $k$-connected.
    \end{enumerate}
\end{proposition}

While Table~\ref{table:all sm3c matroids} lists all super-minimally $3$-connected matroids with at most four elements, in the remainder of the paper, we will focus primarily on those super-minimally $3$-connected matroids with at least five elements. The following is an immediate consequence of Proposition~\ref{prop:small smkc} and a theorem of Oxley~\cite[(4.7)]{Oxley_OMC}.

\begin{lemma}\label{lem:rank at most six}
    If $M$ is a super-minimally $3$-connected matroid with at least four elements and $r(M)\leq 6$, then
    \[|E(M)|\leq 2r(M).\]
\end{lemma}

\section{Structural Lemmas}

In this section, we prove some structural lemmas for super-minimally $3$-connected matroids. The next lemma provides a characterization analogous to Bixby’s Lemma for $3$-connected matroids.

\begin{lemma}\label{lem:sm3c_si_or_co}
    Suppose $M$ is a super-minimally $3$-connected matroid with at least five elements. If $e\in E(M)$, then either $\si(M/e)$ is $3$-connected, or $\co(M\backslash e)$ is super-minimally $3$-connected.
\end{lemma}

\begin{proof}
    Suppose that $\si(M/e)$ is not $3$-connected. Then, for every $2$-separation $\{A,B\}$ of $M/e$, we know $e\in\cl_M(A)\cap\cl_M(B)$. First we show the following.
    \begin{sublemma}\label{sublem:exactly one element}
        If $\{A,B\}$ is a $2$-separation of $M/e$ and $\{e,f,g\}$ is a triad $T^*$ of $M$, then $T^*$ meets each of $A$ and $B$ in exactly one element.
    \end{sublemma}

    Suppose $\{f,g\}\subseteq A$. Because $e\in\cl_M(B)$, there is a circuit $C$ of $M$ in $B\cup\{e\}$ containing $e$. However, $|C\cap T^*|=1$, a contradiction. By symmetry, ~\ref{sublem:exactly one element} follows.\\

    Since $\si(M/e)$ is not $3$-connected, $M/e$ has a nonminimal $2$-separation $\{X,Y\}$.
    \begin{sublemma}\label{sublem:not in the closure}
        If $\{e,f,g\}$ is a triad of $M$ such that $f\in X$ and $g\in Y$, then $f\notin \cl_{M/e}(Y)$ and $g\notin\cl_{M/e}(X)$.
    \end{sublemma}

    To see this, observe that, as $\{X,Y\}$ is a nonminimal $2$-separation, $\min\{|X|,|Y|\}\geq 3$. If $f\in\cl_{M/e}(Y)$, then, by Proposition~\ref{prop:closure coclosure}, $\{X-\{f\},Y\cup\{f\}\}$ is also a $2$-separation of $M/e$. However, $\{f,g\}\subseteq Y\cup\{f\}$, contradicting~\ref{sublem:exactly one element}. By symmetry,~\ref{sublem:not in the closure} follows.\\

    \begin{sublemma}\label{sublem: two triads intersect at e}
        Suppose $T_1^*$ and $T_2^*$ are two distinct triads of $M$ that both contain $e$. Then $T_1^*\cap T_2^*=\{e\}$.
    \end{sublemma}

    Suppose $T_1^*=\{f_1,e,g\}$ and $T_2^*=\{f_2,e,g\}$. Without loss of generality, assume $g\in Y$. Then, by~\ref{sublem:exactly one element}, we have $\{f_1,f_2\}\subseteq X$. By circuit elimination, $\{f_1,f_2,e\}$ is a triad of $M$. But $|\{f_1,f_2,e\}\cap X|=2$, contradicting~\ref{sublem:exactly one element}. This establishes~\ref{sublem: two triads intersect at e}.\\

    Since $\si(M/e)$ is not $3$-connected, by Bixby’s Lemma, it follows that $\co(M\backslash e)$ is $3$-connected. Because $M$ is super-minimally $3$-connected, the matroid $M\backslash e$ contains at least one series pair. Therefore, $M$ has at least one triad containing $e$. Suppose $M$ has exactly $k$ triads $T_1^*,T_2^*,\dots,T_k^*$ containing $e$. By~\ref{sublem:exactly one element} and~\ref{sublem: two triads intersect at e}, we know, for each $i\in\{1,2,\dots,k\}$, the triad $T_i^*$ has the form $\{e,f_i,g_i\}$ such that $f_i\in X$ and $g_i\in Y$. Moreover, $f_i=f_j$ or $g_i=g_j$ if and only if $i=j$ (see Figure~\ref{fig_all_triads} for an illustrative diagram). Thus, we may assume that $\co(M\backslash e)=M\backslash e/\{f_1,f_2,\dots,f_k\}$.
    \begin{center}
    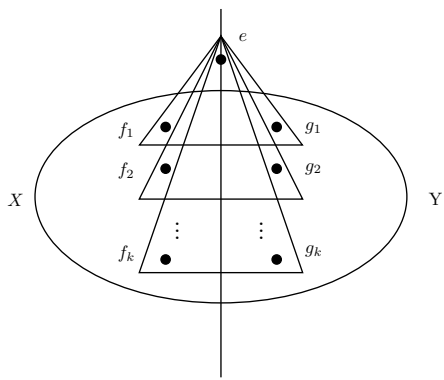
\begin{figure}[htb]
    \hbox to \hsize{
	\hfil
	\resizebox{6cm}{!}{\tikzset{every picture/.style={line width=0.75pt}} 

\begin{tikzpicture}[x=0.75pt,y=0.75pt,yscale=-1,xscale=1]

\draw   (128,144.09) .. controls (128,100.5) and (189.83,65.17) .. (266.1,65.17) .. controls (342.37,65.17) and (404.2,100.5) .. (404.2,144.09) .. controls (404.2,187.67) and (342.37,223) .. (266.1,223) .. controls (189.83,223) and (128,187.67) .. (128,144.09) -- cycle ;
\draw  [fill={rgb, 255:red, 0; green, 0; blue, 0 }  ,fill opacity=1 ] (262.65,42.09) .. controls (262.65,40.18) and (264.2,38.64) .. (266.1,38.64) .. controls (268,38.64) and (269.55,40.18) .. (269.55,42.09) .. controls (269.55,43.99) and (268,45.53) .. (266.1,45.53) .. controls (264.2,45.53) and (262.65,43.99) .. (262.65,42.09) -- cycle ;
\draw    (266.1,4.5) -- (266.1,50.92) -- (266.1,278.33) ;
\draw  [fill={rgb, 255:red, 0; green, 0; blue, 0 }  ,fill opacity=1 ] (221.55,92.22) .. controls (221.55,90.32) and (223.1,88.78) .. (225,88.78) .. controls (226.9,88.78) and (228.45,90.32) .. (228.45,92.22) .. controls (228.45,94.13) and (226.9,95.67) .. (225,95.67) .. controls (223.1,95.67) and (221.55,94.13) .. (221.55,92.22) -- cycle ;
\draw  [fill={rgb, 255:red, 0; green, 0; blue, 0 }  ,fill opacity=1 ] (221.55,123.22) .. controls (221.55,121.32) and (223.1,119.78) .. (225,119.78) .. controls (226.9,119.78) and (228.45,121.32) .. (228.45,123.22) .. controls (228.45,125.13) and (226.9,126.67) .. (225,126.67) .. controls (223.1,126.67) and (221.55,125.13) .. (221.55,123.22) -- cycle ;
\draw  [fill={rgb, 255:red, 0; green, 0; blue, 0 }  ,fill opacity=1 ] (304.05,92.22) .. controls (304.05,90.32) and (305.6,88.78) .. (307.5,88.78) .. controls (309.4,88.78) and (310.95,90.32) .. (310.95,92.22) .. controls (310.95,94.13) and (309.4,95.67) .. (307.5,95.67) .. controls (305.6,95.67) and (304.05,94.13) .. (304.05,92.22) -- cycle ;
\draw  [fill={rgb, 255:red, 0; green, 0; blue, 0 }  ,fill opacity=1 ] (304.05,123.22) .. controls (304.05,121.32) and (305.6,119.78) .. (307.5,119.78) .. controls (309.4,119.78) and (310.95,121.32) .. (310.95,123.22) .. controls (310.95,125.13) and (309.4,126.67) .. (307.5,126.67) .. controls (305.6,126.67) and (304.05,125.13) .. (304.05,123.22) -- cycle ;
\draw  [fill={rgb, 255:red, 0; green, 0; blue, 0 }  ,fill opacity=1 ] (221.55,190.72) .. controls (221.55,188.82) and (223.1,187.28) .. (225,187.28) .. controls (226.9,187.28) and (228.45,188.82) .. (228.45,190.72) .. controls (228.45,192.63) and (226.9,194.17) .. (225,194.17) .. controls (223.1,194.17) and (221.55,192.63) .. (221.55,190.72) -- cycle ;
\draw  [fill={rgb, 255:red, 0; green, 0; blue, 0 }  ,fill opacity=1 ] (304.05,190.72) .. controls (304.05,188.82) and (305.6,187.28) .. (307.5,187.28) .. controls (309.4,187.28) and (310.95,188.82) .. (310.95,190.72) .. controls (310.95,192.63) and (309.4,194.17) .. (307.5,194.17) .. controls (305.6,194.17) and (304.05,192.63) .. (304.05,190.72) -- cycle ;
\draw   (266.1,24.73) -- (326.57,105.61) -- (205.63,105.61) -- cycle ;
\draw   (266.1,24.73) -- (326.78,145.83) -- (205.42,145.83) -- cycle ;
\draw   (266.1,24.73) -- (326.78,200.5) -- (205.42,200.5) -- cycle ;

\draw (106,140.4) node [anchor=north west][inner sep=0.75pt]    {$X$};
\draw (419,139.4) node [anchor=north west][inner sep=0.75pt]    {$\mathrm{Y}$};
\draw (229.55,153.73) node [anchor=north west][inner sep=0.75pt]  [font=\Large]  {$\vdots $};
\draw (291.95,153.73) node [anchor=north west][inner sep=0.75pt]  [font=\Large]  {$\vdots $};
\draw (187.93,87.13) node [anchor=north west][inner sep=0.75pt]    {$f_{1}$};
\draw (187.93,117.47) node [anchor=north west][inner sep=0.75pt]    {$f_{2}$};
\draw (187.93,178.8) node [anchor=north west][inner sep=0.75pt]    {$f_{k}$};
\draw (327.27,87.13) node [anchor=north west][inner sep=0.75pt]    {$g_{1}$};
\draw (327.27,117.47) node [anchor=north west][inner sep=0.75pt]    {$g_{2}$};
\draw (327.27,178.8) node [anchor=north west][inner sep=0.75pt]    {$g_{k}$};
\draw (278,21.4) node [anchor=north west][inner sep=0.75pt]    {$e$};

\end{tikzpicture}}%
	\hfil
    }
    \caption{Every triad of $M$ containing $e$ has the form $\{e,f_i,g_i\}$.}\label{fig_all_triads}
    \end{figure}
    \end{center}

    Suppose $\co(M\backslash e)$ is not super-minimally $3$-connected. Then it has a proper $3$-connected restriction $N$ such that $|E(N)|\geq 4$. Next we show that

    \begin{sublemma}\label{sublem:one of gi}
        $E(N)$ contains $g_i$ for some $i\in\{1,2,\dots,k\}$.
    \end{sublemma}

    Suppose that $E(N)\cap \{g_1,g_2,\dots,g_k\}=\emptyset$. We know that
    \[N=\co(M\backslash e)|E(N)=M\backslash e/\{f_1,f_2,\dots,f_k\}\backslash\{g_1,g_2,\dots,g_k\}|E(N).\]
    Since $\{e,f_i,g_i\}$ is a triad for all $i\in\{1,2,\dots,k\}$, it follows that
    \[N=M\backslash\{e,f_1,f_2,\dots,f_k,g_1,g_2,\dots,g_k\}|E(N),\]
    which is a restriction of $M$, contradicting $M$ being super-minimally $3$-connected. Therefore~\ref{sublem:one of gi} holds.\\

    Let $I=\{i\in\{1,2,\dots,k\}:g_i\in E(N)\}$, and let $F_I=\{f_i:i\in I\}$ and $G_I=\{g_i:i\in I\}$. By~\ref{sublem:one of gi}, without loss of generality, we may assume that $g_1\in G_I$. Let $L$ be $M|(E(N)\cup F_I\cup\{e\})$ and recall that $\{f_1,f_2,\dots,f_k\}\subseteq X$ and $\{g_1,g_2,\dots,g_k\}\subseteq Y$. We show next that

    \begin{sublemma}\label{sublem:e is not a coloop}
        $e$ is not a coloop of $L$.
    \end{sublemma}

    It is clear that $N=L\backslash e/F_I$. Suppose $e$ is a coloop of $L$. Then
    \[N=L/e/F_I=(M/e/F_I)|E(N),\]
    which is a $3$-connected minor of $M/e$ with at least four elements. Recall that $\{X,Y\}$ is a $2$-separation of $M/e$. Then $\min\{|E(N)\cap X|,|E(N)\cap Y|\}\leq 1$. 
    
    By~\ref{sublem:one of gi}, $|E(N)\cap Y|\geq |G_I|\geq 1$. Suppose $|E(N)\cap Y|=1$. Then $E(N)\cap Y=\{g_1\}$ and $g_1\in\cl_N(E(N)\cap X)$. Because $F_I\subseteq X$, it follows that $g_1\in \cl_{M/e}(X)$, contradicting~\ref{sublem:not in the closure}.

    We now know that $|E(N)\cap X|\leq 1$ and $|E(N)\cap Y|\geq 3$. Because $e$ is a coloop of $L$ and $N$ is $3$-connected, $f_i$ and $g_i$ are in series in $L\backslash e$ for all $i\in I$. Therefore,
    \[L/ e/((F_I\cup\{g_1\})-\{f_1\})\]
    is a $3$-connected minor $N'$ of $M/e$ that is isomorphic to $N$. Because $\{X,Y\}$ is a $2$-separation of $M/e$, we know $\min\{|E(N')\cap X|,|E(N')\cap Y|\}\leq 1$. Clearly, $|E(N')\cap X|\geq 1$ and $|E(N')\cap Y|\geq 2$. Hence $E(N')\cap X=\{f_1\}$. If $F_I=\{f_1\}$, then $f_1\in\cl_{N'}(E(N')\cap Y)$. Because $(F_I\cup\{g_1\})-\{f_1\}=\{g_1\}$ and $g_1\in Y$, it follows that $f_1\in\cl_{M/e}(Y)$, contradicting~\ref{sublem:not in the closure}. Thus $F_I\neq \{f_1\}$. Without loss of generality, we may assume that $f_2\in F_I$. Let $N''$ be the matroid
    \[L/e/((F_I\cup\{g_1,g_2\})-\{f_1,f_2\}).\]
    Clearly, $N''\cong N$. It is straightforward to check that $|E(N'')\cap X|=2$ and $|E(N'')\cap Y|\geq 2$, a contradiction as $N$ is $3$-connected. We conclude that~\ref{sublem:e is not a coloop} holds. \\

    In the original matroid $M$, the set $\{e,f_i,g_i\}$ is a triad for all $i\in\{1,2,\dots,k\}$. The restriction $L$ of $M$ contains $e$, but neither $f_i$ nor $g_i$ if $i\in\{1,2,\dots,k\}-I$. If $\{1,2,\dots,k\}-I\neq\emptyset$, then the element $e$ will be a coloop in $L$, contradicting~\ref{sublem:e is not a coloop}. Therefore, $I=\{1,2,\dots,k\}$. We show next that

    \begin{sublemma}\label{sublem:L is 3-con}
        $L$ is $3$-connected.
    \end{sublemma}

    Because $L\backslash e$ has $\{f_i,g_i\}$ as a series pair for all $i\in I$, and $L\backslash e/F_I=N$, which is a $3$-connected matroid with at least four elements, it is clear that $L\backslash e$ is $2$-connected. Moreover, since $e$ is neither a loop nor a coloop, $L$ is also $2$-connected. Suppose $\{C,D\}$ is a $2$-separation of $L$. Because the series pairs in $L\backslash e$ are precisely $\{f_i,g_i\}$ for all $i\in I$ and each $\{e,f_i,g_i\}$ is a triad in $L$, we know $L$ has no series pairs. Clearly, $L$ has no parallel pairs either, so $\min\{|C|,|D|\}\geq 3$. Without loss of generality, we may assume that $e\in C$. By Proposition~\ref{prop:closure coclosure}, we may also assume that, for each $i\in I$, the set $\{f_i,g_i\}$ is contained in either $C$ or $D$. Let $s=|C\cap F_I|$ and $t=|D\cap F_I|$. Clearly $s+t=|I|=k$. Because $N=L\backslash e/F_I$, we see that
    \begin{align*}
        r_N(C-\{e\}-F_I)+r_N(D-F_I)-r(N)&\leq (r_L(C)-s)+(r_L(D)-t)-(r(L)-k)\\
        &=r_L(C)+r_L(D)-r(L)\\
        &\leq 1.
    \end{align*}
    Moreover, $|C-\{e\}-F_I|\geq \max\{2-s,s\}$ and $|D-F_I|\geq \max\{3-t,t\}$. Since $N$ is $3$-connected and $\{C-\{e\}-F_I, D-F_I\}$ is not a $2$-separation, the only possibility is that $C=\{e,f_j,g_j\}$ for some $j\in I$. However, $\{e,f_j,g_j\}$ is a triad and $L$ is not isomorphic to $U_{2,4}$, so $r_L(C)+r_L^*(C)-|C|=2$. Thus $\{C,D\}$ is not a $2$-separation, a contradiction. Therefore~\ref{sublem:L is 3-con} holds.\\

    Because $|E(L)|\geq |E(N)|\geq 4$, we see that $L$ is a $3$-connected proper restriction of $M$ with at least four elements. As $M$ is super-minimally $3$-connected, this contradiction finishes the proof. 
\end{proof}

\begin{lemma}\label{lem:reduce triangle}
    Suppose $M$ is a super-minimally $3$-connected matroid with at least seven elements and $\{e,f,g\}$ is a triangle of $M$. Then $M$ has an element $x$ of $\{e,f,g\}$ such that $\co(M\backslash x)$ is super-minimally $3$-connected.
\end{lemma}

\begin{proof}
   Since $|E(M)|\ge7$, Proposition~\ref{prop:small smkc} implies that $M$ is minimally $3$-connected. Consequently, none of $M\backslash e$, $M\backslash f$, or $M\backslash g$ is $3$-connected. After possibly relabeling $e$, $f$, and $g$, we deduce by Tutte’s Triangle Lemma that there is an element $h$ such that $\{e,f,h\}$ is a triad of $M$. Now we show that \begin{sublemma}\label{sublem:M/g is not 3-con} 
        $\si(M/g)$ is not $3$-connected. 
   \end{sublemma}
    Suppose $\si(M/g)$ is $3$-connected. Let $H=E(M)-\{e,f,h\}$, which is a hyperplane of $M$ (see Figure~\ref{fig_noncontractible}). Since $|E(M)|\ge7$, we have $|H|\ge4$. Because $M$ is simple and has no $U_{2,4}$-restriction, $r(H)\ge3$, and hence $r(M)\ge4$. Therefore $r(M/g)=r(\si(M/g))\geq 3$.  By Table~\ref{table:all sm3c matroids}, we know $\si(M/g)$ has at least five elements. Note that $e$ and $f$ are parallel in $M/g$. In the simplification, deleting one of them (say $f$) leaves $e$ and $h$ in series in $\si(M/g)$. Since $\si(M/g)$ has at least five elements, $\{e,h\}$ is a $2$-separating set, a contradiction. Therefore~\ref{sublem:M/g is not 3-con} holds.

   \begin{center}
    \begin{figure}[htb]
    \hbox to \hsize{
	\hfil
	\resizebox{4cm}{!}{\tikzset{every picture/.style={line width=0.75pt}} 

\begin{tikzpicture}[x=0.75pt,y=0.75pt,yscale=-1,xscale=1]

\draw  [fill={rgb, 255:red, 0; green, 0; blue, 0 }  ,fill opacity=1 ] (273.38,107.41) .. controls (273.38,105) and (275.33,103.05) .. (277.74,103.05) .. controls (280.15,103.05) and (282.1,105) .. (282.1,107.41) .. controls (282.1,109.82) and (280.15,111.77) .. (277.74,111.77) .. controls (275.33,111.77) and (273.38,109.82) .. (273.38,107.41) -- cycle ;
\draw  [fill={rgb, 255:red, 0; green, 0; blue, 0 }  ,fill opacity=1 ] (221.27,106.12) .. controls (221.27,103.71) and (223.22,101.76) .. (225.63,101.76) .. controls (228.03,101.76) and (229.99,103.71) .. (229.99,106.12) .. controls (229.99,108.53) and (228.03,110.48) .. (225.63,110.48) .. controls (223.22,110.48) and (221.27,108.53) .. (221.27,106.12) -- cycle ;
\draw  [dash pattern={on 4.5pt off 4.5pt}]  (200.06,6.83) -- (200.06,206.72) ;
\draw  [draw opacity=0] (199.86,173.44) .. controls (192.1,176.55) and (183.63,178.26) .. (174.76,178.26) .. controls (137.45,178.26) and (107.2,148.01) .. (107.2,110.7) .. controls (107.2,73.39) and (137.45,43.14) .. (174.76,43.14) .. controls (183.87,43.14) and (192.57,44.94) .. (200.5,48.22) -- (174.76,110.7) -- cycle ; \draw   (199.86,173.44) .. controls (192.1,176.55) and (183.63,178.26) .. (174.76,178.26) .. controls (137.45,178.26) and (107.2,148.01) .. (107.2,110.7) .. controls (107.2,73.39) and (137.45,43.14) .. (174.76,43.14) .. controls (183.87,43.14) and (192.57,44.94) .. (200.5,48.22) ;  
\draw    (200.06,106.78) -- (277.74,107.41) ;
\draw    (199.56,159.28) -- (277.74,107.41) ;
\draw  [fill={rgb, 255:red, 0; green, 0; blue, 0 }  ,fill opacity=1 ] (234.29,133.35) .. controls (234.29,130.94) and (236.24,128.99) .. (238.65,128.99) .. controls (241.06,128.99) and (243.01,130.94) .. (243.01,133.35) .. controls (243.01,135.75) and (241.06,137.71) .. (238.65,137.71) .. controls (236.24,137.71) and (234.29,135.75) .. (234.29,133.35) -- cycle ;
\draw  [fill={rgb, 255:red, 0; green, 0; blue, 0 }  ,fill opacity=1 ] (195.2,159.28) .. controls (195.2,156.87) and (197.15,154.92) .. (199.56,154.92) .. controls (201.96,154.92) and (203.92,156.87) .. (203.92,159.28) .. controls (203.92,161.69) and (201.96,163.64) .. (199.56,163.64) .. controls (197.15,163.64) and (195.2,161.69) .. (195.2,159.28) -- cycle ;
\draw    (199.12,52.88) -- (238.65,133.35) ;
\draw    (199.12,52.88) -- (277.74,107.41) ;
\draw    (199.56,159.28) -- (237.46,80.5) ;

\draw (210.06,164.52) node [anchor=north west][inner sep=0.75pt]    {$g$};
\draw (281.86,118.73) node [anchor=north west][inner sep=0.75pt]    {$e$};
\draw (242.77,144.66) node [anchor=north west][inner sep=0.75pt]    {$f$};
\draw (235.65,94.02) node [anchor=north west][inner sep=0.75pt]    {$h$};
\draw (146,107.4) node [anchor=north west][inner sep=0.75pt]    {$H$};

\end{tikzpicture}}%
	\hfil
    }
    \caption{An illustrative geometric representation of $M$}\label{fig_noncontractible}
    \end{figure}
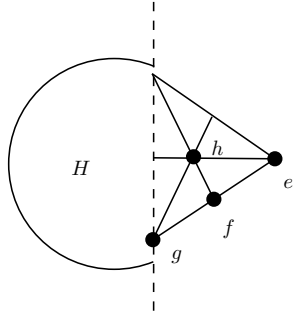
    \end{center}

    By Lemma~\ref{lem:sm3c_si_or_co}, it follows that $\co(M\backslash g)$ is super-minimally $3$-connected, which completes the proof.
    \end{proof}

    The next result extends a result of Oxley~\cite[(5.1)]{Oxley_OMC}.

    \begin{lemma}
        Suppose $M$ is a minimally $3$-connected matroid such that $M\backslash x/y$ is  isomorphic to $M(\W_k)$ or $\W^k$ for some $k\geq 4$. Then $M$ is isomorphic to $M(\W_{k+1})$ or $\W^{k+1}$.
    \end{lemma}
    
    \begin{proof}
        Because both $M$ and $M\backslash x/y$ are $3$-connected, and $M\backslash x$ is not $3$-connected, by Lemma~\ref{lemma:oxley contraction}, $x$ and $y$ are in a triad $\{x,y,z\}$ of $M$. We note that $M\backslash x$ is a series-extension of $\W_k$ or $\W^k$. Let
        \[E(M\backslash x/y)=\{a_1,b_1,a_2,b_2,\dots,a_k,b_k\},\]
        for some integer $k\geq 4$, such that the triangles of $M\backslash x/y$ are exactly 
        \[\{b_i,a_{i+1},b_{i+1}\}\quad (i\in\mathbb{Z}_k),\]
        and the triads of $M\backslash x/y$ are exactly 
        \[\{a_i,b_i,a_{i+1}\}\quad (i\in\mathbb{Z}_k).\]
        We know, in $M$, for each $i\in\mathbb{Z}_k$, either
        \[\{b_i,a_{i+1},b_{i+1}\}\quad \text{or }\quad \{b_i,a_{i+1},b_{i+1},y\}\]
        is a circuit, and either
        \[\{a_i,b_i,a_{i+1}\}\quad \text{or }\quad \{a_i,b_i,a_{i+1},x\}\]
        is a cocircuit. Therefore, it is elementary to see the following.

        \begin{sublemma}\label{sublem:all triads}
            If $T^*$ is a triad of $M$, then
            
            \begin{enumerate}
                \item[(1)] $T^*=\{a_i,b_i,a_{i+1}\}$ for some $i\in\mathbb{Z}_k$, or
                \item[(2)] $T^*=\{x,y,z\}$, or
                \item[(3)] $T^*$ contains $y$ but not $x$ and, in particular, $T^*$ is also a triad of $M\backslash x$.
            \end{enumerate}
        \end{sublemma}

        Because of the symmetries of $M(\W_k)$ and $\W^k$, it suffices to consider the following two cases:
        
        \begin{enumerate}
            \item[(\romannum{1})] $z=b_1$, or
            \item[(\romannum{2})] $z=a_1$. 
        \end{enumerate}

        Suppose that (\romannum{1}) holds, that is, $\{x,y,b_1\}$ is a triad in $M$. For each $i\in\{2,3,\dots,k-1\}$, because $|\{b_i,a_{i+1},b_{i+1},y\}\cap\{x,y,b_1\}|=1$, we know $\{b_i,a_{i+1},b_{i+1}\}$ is a triangle in $M$. Next we show that
        \begin{sublemma}\label{sublem:2 to k triads}
            $\{a_i,b_i,a_{i+1}\}$ is a triad of $M$, for each $i\in\{2,3,\dots,k\}$.
        \end{sublemma}

        For each $i\in\{2,3,\dots,k-1\}$, we know $b_i$ is in the triangle $\{b_i,a_{i+1},b_{i+1}\}$. Since $M$ is minimally $3$-connected, by Tutte's Triangle Lemma, $M$ has a triad $T^*$ containing $b_i$ and exactly one of $\{a_{i+1},b_{i+1}\}$. If $y\in T^*$, then $(T^*\cup \{x,y,z\})-\{y\}$ contains a cocircuit containing $x$. Thus $M\backslash x/y$ has $(T^*-\{y\})\cup\{z\}$ as a triad, which contradicts~\ref{sublem:all triads}. Hence, by~\ref{sublem:all triads}, $\{a_i,b_i,a_{i+1}\}$ is a triad in $M$ for all $i\in\{2,3,\dots,k-1\}$. Moreover, it follows by symmetry that $\{a_k,b_k,a_1\}$ is a triad. Hence~\ref{sublem:2 to k triads} holds. \\
        
        Let $\{b_1,b_2,\dots,b_k,y\}$ be a basis $B_y$ of $M$ and let $C_x$ be the fundamental circuit of $x$ with respect to $B_y$. Because $|E(M)|\geq 10$, we know $C_x\neq \{x,y,b_1\}$, otherwise
        \[\lambda_M(\{x,y,b_1\})=r_M(\{x,y,b_1\})+r^*_M(\{x,y,b_1\})-|\{x,y,b_1\}|=1,\]
        which implies that $\{x,y,b_1\}$ is a $2$-separating set, a contradiction. Therefore, $C_x$ contains $b_j$ for some $j\neq 1$. Since $|C_x\cap \{a_j,b_j,a_{j+1}\}|=1$, we deduce that $\{a_j,b_j,a_{j+1}\}$ is not a cocircuit, a contradiction to~\ref{sublem:2 to k triads}. 

        We may now suppose that (\romannum{1}) fails and (\romannum{2}) holds, that is, $\{x,y,a_1\}$ is a triad. For each $i\in\{1,2,\dots, k-1\}$, since $|\{b_i,a_{i+1},b_{i+1},y\}\cap\{x,y,a_1\}|=1$, we deduce that $\{b_i,a_{i+1},b_{i+1}\}$ is a triangle of $M$. Next we show that
        
        \begin{sublemma}\label{sublem:2 to k-1 triads (case 2)}
            $\{a_i,b_i,a_{i+1}\}$ is a triad of $M$, for each $i\in\{2,3,\dots,k-1\}$.
        \end{sublemma}

        Since $\{b_i,a_{i+1},b_{i+1}\}$ is a triangle of $M$ and $M$ is minimally $3$-connected, by Tutte's Triangle Lemma, $b_i$ is in a triad containing exactly one of $a_{i+1}$ and $b_{i+1}$. Because $\{b_j,a_{j+1},b_{j+1}\}$ is a triangle of $M$, for each $j\in\{1,2,\dots, k-1\}$, it is straightforward to check that neither $\{b_i,a_{i+1},y\}$ nor $\{b_i,b_{i+1},y\}$ is a triad of $M$. Therefore, by~\ref{sublem:all triads},  $\{a_i,b_i,a_{i+1}\}$ is a triad. Hence~\ref{sublem:2 to k-1 triads (case 2)} holds. 

        \begin{sublemma}\label{sublem:b1a2a3b3 is a circuit}
            $M$ has $\{b_1,a_2,a_3,b_3\}$ as a circuit.
        \end{sublemma}

        As $\{b_1,a_2,b_2\}$ and $\{b_2,a_3,b_3\}$ are circuits, $M$ has a circuit $C$ containing $b_1$ that is contained in $\{b_1,a_2,a_3,b_3\}$. As $\{a_2,b_2,a_3\}$ and $\{a_3,b_3,a_4\}$ are cocircuits, it follows, by orthogonality, that $C=\{b_1,a_2,a_3,b_3\}$. Thus~\ref{sublem:b1a2a3b3 is a circuit} holds.

        \begin{sublemma}\label{sublem:one of the two triads}
            $M$ has $\{a_1,b_1,a_2\}$ or $\{y,b_1,a_2\}$ as a triad.
        \end{sublemma}

        Since $\{b_1,a_2,b_2\}$ is a triangle, by Tutte's Triangle Lemma, $M$ has a triad $T^*$ containing $b_1$ and exactly one element of $\{a_2,b_2\}$. By~\ref{sublem:all triads}, $T^*$ is one of $\{a_1,b_1,a_2\}$, $\{y,b_1,b_2\}$, or $\{y,b_1,a_2\}$. As $\{b_1,a_2,a_3,b_3\}$ is a circuit of $M$, it follows, by orthogonality, that $T^*$ is $\{a_1,b_1,a_2\}$ or $\{y,b_1,a_2\}$.\\

        Let $\{b_1,b_2,\dots,b_k,x\}$ be a basis $B_x$ of $M$. Let $C_{a_1}$ be the fundamental circuit of $a_1$ with respect to $B_x$, and let $C_y$ be the fundamental circuit of $y$ with respect to $B_x$. By~\ref{sublem:2 to k-1 triads (case 2)}, we deduce that $C_{a_1}$ is a subset of $\{a_1,x,b_1,b_k\}$ and $C_y$ is a subset of $\{y,x,b_1,b_k\}$. Because $M$ is simple and has $\{a_1,b_1,a_2\}$ or $\{y,b_1,a_2\}$ as a triad, at least one of $C_{a_1}$ and $C_y$ is a triangle. Moreover, since $\{x,y,a_1\}$ is a triad of $M$, it is clear that $x$ is contained in both $C_{a_1}$ and $C_y$. Recall that, for each of $i\in\{1,2,\dots, k-1\}$, the set $\{b_i,a_{i+1},b_{i+1}\}$ is a triangle of $M$. Thus, there is at most one element of $M$ that is not contained in a triangle. Since $M$ is minimally $3$-connected, there is at most one element of $M$ that is nonessential. By Theorem~\ref{thm:two non-essential}, $M$ is a whirl or the cycle matroid of a wheel.
    \end{proof}

    Combining the last lemma with Lemma~\ref{lemma:small rank wheel and whirl}, we immediately obtain the following result.

    \begin{corollary}\label{cor:growing wheel and whirl}
        Suppose $M$ is a minimally $3$-connected matroid such that $M\backslash x/y$ is  $M(\W_k)$ or $\W^k$ for some $k\geq 2$. Then $M$ is either $M(\W_{k+1})$ or $\W^{k+1}$.
    \end{corollary}

\section{Brittle matroids}
A simple graph $G$ is {\it fragile} if $G$ has no $3$-connected subgraphs. Mader~\cite{Mader} showed that every fragile graph is $4$-degenerate and thus $5$-colorable. Recently, Bonnet et al.~\cite{Bonnet} proved that every fragile graph is $4$-colorable, and they also showed the following~\cite[3.1]{Bonnet}. 

\begin{proposition}
    If $G$ is a fragile graph, then 
    \[|E(G)|\leq 2.5 |V(G)|-5.\]
\end{proposition}

In this section, we extend the concept of fragile graphs to matroids and establish several lemmas that will be used in subsequent proofs. Since the term ``fragile'' has been used with a different meaning in matroid theory, we will instead use the term ``brittle'' throughout this paper. A simple matroid $M$ is {\it brittle} if it has no $3$-connected restriction with at least four elements.

\begin{lemma}
    Suppose $M$ is a brittle matroid. Then one of the following holds.
    
    \begin{enumerate}
        \item[(\romannum{1})] $M = M_1 \oplus M_2$ for some brittle matroids $M_1$ and $M_2$.
        \item[(\romannum{2})] $M = M_3 \oplus_2 M_4$ at basepoint $p$, for some $2$-connected matroids $M_3$ and $M_4$, such that $|E(M_i)|\geq 3$ and $M_i\backslash p$ is brittle for each $i\in\{3,4\}$.
    \end{enumerate}
\end{lemma}

\begin{proof}
    If $M$ is not $2$-connected, then $M = M_1 \oplus M_2$ for some matroids $M_1$ and $M_2$. Since both $M_1$ and $M_2$ are restrictions of $M$, they are also brittle. Hence (\romannum{1}) holds.

    Now assume that $M$ is $2$-connected. As $M$ is brittle and hence not $3$-connected, there are $2$-connected matroids $M_3$ and $M_4$ such that $M = M_3\oplus_2 M_4$ such that $|E(M_3)|\geq 3$ and $|E(M_4)|\geq 3$ (see, for example,~\cite[Theorem~8.3.1]{Oxl11}). Let $p$ be the basepoint of this $2$-sum. Since both $M_3\backslash p$ and $M_4\backslash p$ are restrictions of $M$, they are also brittle. Therefore, (\romannum{2}) holds.
\end{proof}

Observe that $M_1$, $M_2$, $M_3 \backslash p$, and $M_4 \backslash p$ in the previous lemma are all restrictions of $M$. Therefore, the next result follows immediately.

\begin{corollary}\label{cor:tri-free brittle}
    Suppose $M$ is a triangle-free brittle matroid. Then one of the following holds.
    
    \begin{enumerate}
        \item[(\romannum{1})] $M = M_1 \oplus M_2$ for some triangle-free brittle matroids $M_1$ and $M_2$.
        \item[(\romannum{2})] $M = M_3 \oplus_2 M_4$ at basepoint $p$, for some $2$-connected matroids $M_3$ and $M_4$, such that $|E(M_i)|\geq 3$ and $M_i\backslash p$ is triangle-free and brittle for each $i\in\{3,4\}$.
    \end{enumerate}
\end{corollary}

The following is a matroid analogue of a result of Bonnet et al.~\cite[3.2]{Bonnet}.

\begin{lemma}
    If $M$ is a triangle-free brittle matroid, then
    \[
        |E(M)| \leq 2r(M) - 2,
    \]
    unless $M \cong U_{1,1}$.
\end{lemma}

\begin{proof}
    We proceed by induction on $|E(M)|$. When $|E(M)| = 1$, the only brittle matroid with one element is $U_{1,1}$, and
    \[
        |E(U_{1,1})| = 2r(U_{1,1}) - 1.
    \]
    When $|E(M)| = 2$, the only brittle matroid with two elements is $U_{2,2}$, and clearly
    \[
        |E(U_{2,2})| = 2r(U_{2,2}) - 2.
    \]
    Hence, the result holds for $|E(M)| \in\{1,2\}$.

    Now assume that $|E(M)| \ge 3$, and that the statement holds for all triangle-free brittle matroids with fewer than $|E(M)|$ elements. In particular, for every such matroid $N$, we have
    \begin{sublemma}\label{sublem:2r-1}
        $|E(N)| \le 2r(N) - 1$.
    \end{sublemma}

    Suppose first that $M$ is not $2$-connected. Then $M = M_1 \oplus M_2$ for some triangle-free brittle matroids $M_1$ and $M_2$. Since $|E(M_i)| < |E(M)|$ for each $i \in \{1,2\}$, the inductive hypothesis gives
    \begin{align*}
        |E(M)| &= |E(M_1)| + |E(M_2)| \\
        &\le (2r(M_1) - 1) + (2r(M_2) - 1) \tag{by~\ref{sublem:2r-1}}\\
        &= 2r(M) - 2.
    \end{align*}

    Next, suppose that $M$ is $2$-connected. By Corollary~\ref{cor:tri-free brittle}, we may assume $M=M_3 \oplus_2 M_4$ at basepoint $p$, where $M_3$ and $M_4$ are $2$-connected matroids such that $|E(M_i)| \ge 3$ and $M_i \backslash p$ is triangle-free and brittle for each $i \in \{3,4\}$. Clearly, $|E(M_i)| < |E(M)|$, and since $|E(M_i)| \ge 3$, neither $M_3 \backslash p$ nor $M_4 \backslash p$ is isomorphic to $U_{1,1}$. Thus, by the inductive hypothesis,
    
    \begin{sublemma}\label{sublem:2r-2}
        $|E(M_i \backslash p)| \le 2r(M_i \backslash p) - 2$ for each $i \in \{3,4\}$.
    \end{sublemma}
    
    \noindent Therefore,
    \begin{align*}
        |E(M)| &= |E(M_3)| + |E(M_4)| - 2 \\
        &= |E(M_3 \backslash p)| + |E(M_4 \backslash p)| \\
        &\le 2r(M_3 \backslash p) + 2r(M_4 \backslash p) - 4 \tag{by~\ref{sublem:2r-2}}\\
        &= 2r(M_3) + 2r(M_4) - 4 \tag{as $M_3$ and $M_4$ are $2$-connected}\\
        &= 2r(M) - 2.
    \end{align*}
    This completes the proof.
\end{proof}

With Proposition~\ref{prop:small smkc} and Table~\ref{table:all sm3c matroids}, one may check that if $M$ is a super-minimally $3$-connected matroid with at least four elements, then $M\backslash e$ is brittle for all $e\in E(M)$. The next result follows immediately from the previous lemma.

\begin{corollary}\label{cor:triangle-free sm3c}
    If $M$ is a triangle-free super-minimally $3$-connected matroid with at least four elements, then
    \[|E(M)| \leq 2r(M) - 1.\]
\end{corollary}

\section{Density of super-minimally $3$-connected matroids}
In this section, we prove a tight upper bound on the number of elements in a rank-$r$ super-minimally $3$-connected matroid.

\begin{proof}[Proof of Theorem~\ref{thm:density_sm3c_matroids}]
Suppose there is a super-minimally $3$-connected matroid $M$ with at least four elements such that
\[
    |E(M)| \ge 2r(M),
\]
and that $M$ is not isomorphic to a whirl of rank at least two or the cycle matroid of a wheel of rank at least three. Choose such a matroid of minimum rank. By Proposition~\ref{prop:small smkc}, $M$ is a minimally $3$-connected matroid. Thus, by Lemma~\ref{lem:rank at most six} and Lemma~\ref{lemma:small rank wheel and whirl}, $r(M)\geq 6$ and hence $|E(M)|\geq 12$. Suppose $M$ is triangle-free. By Corollary~\ref{cor:triangle-free sm3c}, we know $|E(M)|\leq 2r(M)-1$, a contradiction. Therefore, we may assume that $M$ has a triangle $\{e,f,g\}$. By Lemma~\ref{lem:reduce triangle}, $M$ has an element $x$ in $\{e,f,g\}$ such that $\co(M\backslash x)$ is super-minimally $3$-connected. Note that, since $M$ is super-minimally $3$-connected, $\co(M\backslash x)=M\backslash x/\{y_1,y_2,\dots,y_k\}$ for some $k\geq 1$. If $k\geq 2$, then $|E(\co(M\backslash x))|>2r(\co(M\backslash x))$ and $r(\co(M\backslash x))<r(M)$, contradicting the minimality of $M$. Thus, we may assume that $\co(M\backslash x)=M\backslash x/y$. Then,
\begin{align*}
    |E(M\backslash x/y)|&=|E(M)|-2\geq 2r(M)-2=2r(M\backslash x/y).
\end{align*}
Since $M\backslash x/y$ is super-minimally $3$-connected, the choice of $M$ means that $M\backslash x/y$ is either a whirl or the cycle matroid of a wheel. By Corollary~\ref{cor:growing wheel and whirl}, $M$ is either a whirl or the cycle matroid of a wheel, a contradiction.
\end{proof}

\section{Triads in super-minimally $3$-connected matroids}

In addition to determining the extremal density of minimally $3$-connected graphs, Halin~\cite[Satz~6]{Halin-German} proved that every such graph $G$ has at least $\frac{2|V(G)|+6}{5}$ vertices of degree three, and he showed that this bound is sharp. For super-minimally $3$-connected graphs, Ge~\cite[Theorem~1.7]{Ge_sm3c} established the analogous sharp bound by showing that every such graph $G$ has at least $\frac{|V(G)|+3}{2}$ vertices of degree three.

For matroids, triads play an analogous role to vertices of degree three in graphs. Lemos~\cite{Lemos2004,Lemos2007} proved the following results concerning triads in minimally $3$-connected matroids.

\begin{theorem}
Let $M$ be a minimally $3$-connected matroid with at least eight elements. Then the number of elements contained in at least one triad is at least
\[
\frac{5|E(M)|+30}{9}.
\]
\end{theorem}

\begin{theorem}
Let $M$ be a minimally $3$-connected matroid with at least four elements. Then $M$ has at least
\[
\frac{r(M)+6}{4}
\]
triads.
\end{theorem}

Both bounds are sharp with extremal examples given in~\cite[p.~172]{Lemos2004} and~\cite[p.~940]{Lemos2007}, respectively. It is straightforward to verify that these extremal examples are not only minimally $3$-connected but also super-minimally $3$-connected. Thus, using Proposition~\ref{prop:small smkc}, we obtain the following analogous sharp bounds for super-minimally $3$-connected matroids.

\begin{corollary}
Let $M$ be a super-minimally $3$-connected matroid with at least eight elements. Then the number of elements contained in at least one triad is at least
\[
\frac{5|E(M)|+30}{9}.
\]
\end{corollary}

\begin{corollary}
Let $M$ be a super-minimally $3$-connected matroid with at least four elements. Then $M$ has at least
\[
\frac{r(M)+6}{4}
\]
triads.
\end{corollary}

\end{document}